\documentclass[hyperref]{amsart}
\usepackage{etoolbox}

\ifdef{\ebook}{
  \usepackage[paperwidth=9cm, paperheight=12cm, hmargin={0.17in, 0.17in}, vmargin={0.50in, 0.17in}]{geometry} \usepackage{kmath,kerkis}
}{
\usepackage[verbose,tmargin=1in,lmargin=1.35in,rmargin=1.35in]{geometry}
}

\usepackage{amssymb}
\usepackage{amsmath}
\usepackage{xypic}
\usepackage{nicefrac}
\usepackage{color}

\usepackage{multirow}
\usepackage{bigdelim}
\ifdef{\ebook}{\usepackage{breqn}}{}
\usepackage{hyperref}

\newtheorem{theorem}{Theorem}[section]
\newtheorem*{ttheorem}{Tian's criterion}
\newtheorem{lemma}[theorem]{Lemma}
\newtheorem{proposition}[theorem]{Propositio}
\newtheorem*{conjecture}{Conjecture}

{\theoremstyle{remark}
\newtheorem*{remark}{Remark}

}

\newtheorem{maintheorem}[theorem]{Theorem}

\theoremstyle{definition}

\newtheorem{definition}[theorem]{Definition}
\newtheorem{example}[theorem]{Example}

\newcommand{\pt}{\{\mathbf{pt}\}}
\renewcommand{\Vert}{\mathcal{V}}
\newcommand{\Hor}{\mathcal{H}}
\newcommand{\QGIT}{{\!/\!\!/}}
\newcommand{\tXn}{{\widetilde{X}^\circ}}
\newcommand{\Xn}{{X^\circ}}
\newcommand{\tX}{{\widetilde{X}}}
\newcommand{\tD}{{\widetilde{D}}}
\newcommand{\T}{{\mathcal{T}}}
\newcommand{\tY}{{\widetilde{Y}}}
\newcommand{\CC}{\mathbf{C}}
\renewcommand{\L}{\mathcal{L}}

\newcommand{\B}{\mathcal{B}}
\newcommand{\QQ}{\mathbf{Q}}
\newcommand{\ZZ}{\mathbf{Z}}

\newcommand{\PP}{\mathbf{P}}
\newcommand{\A}{\mathbf{A}}

\newcommand{\CO}{{\mathcal{O}}}

\newcommand{\X}{{\mathcal{X}}}

\newcommand{\sufficient}{valuable }

\DeclareMathOperator{\Sl}{Sl}

\DeclareMathOperator{\Bl}{Bl}

\DeclareMathOperator{\im}{im}

\DeclareMathOperator{\Hom}{Hom}
\DeclareMathOperator{\divisor}{div}

\DeclareMathOperator{\wdiv}{Div}

\DeclareMathOperator{\spec}{Spec}

\DeclareMathOperator{\supp}{supp}

\DeclareMathOperator{\Aut}{Aut}
\DeclareMathOperator{\lct}{g\bf{lct}}

\title{K\"ahler-Einstein metrics on symmetric Fano T-varieties}

\author[H.~S\"u\ss{}]{Hendrik S\"u\ss{}}
\address{Departement of Higher Algebra, Faculty of Mechanics and Mathematics, Moscow State University, Leninskie Gory 1, GSP-1, Moscow, 119991, Russia}
\email{suess@sdf-eu.org}
\urladdr{http://suess.sdf-eu.org/}
\thanks{I was supported by a fellowship of the Alexander von Humboldt Foundation}

\subjclass[2010]{32Q20 (Primary) 14L30, 14J45 (Secondary)}
\keywords{Kahler-Einstein metric, T-variety, torus action, Fano variety, log-canonical threshold}

\begin{document}
\maketitle
\begin{abstract}
  We relate the global log canonical threshold of a variety with torus action to the global log canonical threshold of its quotient. We apply this to certain Fano varieties and use Tian's criterion to prove the existence of K\"ahler-Einstein metrics on them. In particular, we obtain simple examples of Fano threefolds being K\"ahler-Einstein but admitting deformations without K\"ahler-Einstein metric. 
\end{abstract}

\section{Introduction}
This paper deals with the question of the existence of K\"ahler-Einstein metrics on Fano manifolds. While for the case of negative Ricci curvature and for the Ricci flat case the existence of K\"ahler-Einstein metrics is known for some time due to famous results of Aubin and Yau, the Fano case is still not completely understood. There are some known obstructions such as the reductivity of the automorphism group and the vanishing of the Futaki character \cite{0506.53030}. There is also a sufficient criterion due to Tian \cite{0599.53046}, but for a long time there was no complete algebraic characterization of the K\"ahler-Einstein property, only recently it was proved that a certain notion of stability plays this r\^ole \cite{2012arXiv1210.7494C,2012arXiv1211.4669T,2013arXiv1302.0282C}. And even this stability condition is hard to check in general.

However, for toric varieties the problem is completely solved. By a result of Wang and Zhu a toric Fano variety is K\"ahler-Einstein if and only if the Futaki character vanishes \cite{1086.53067}. Moreover, the Futaki character can be easily calculated as the barycenter of the polytope corresponding to the toric Fano manifold \cite{0661.53032}. These positive results suggest to also exploit lower dimensional torus actions for approaching the problem. 
Since Wang's and Zhu's result seems not to generalize well for non-toric situations, we follow the idea of an older paper by Batyrev and Selivanova \cite{0939.32016}. They proved the existence of K\"ahler-Einstein metrics on \emph{symmetric} toric Fano manifolds. It turns out that this result generalizes well. More precisely, high symmetry of the Fano manifold allows us to approach the question by solving a related problem on a torus quotient.

One motivation for considering lower dimensional torus actions in this context is the wish to study the behavior of the K\"ahler-Einstein property under deformations of the manifold, similar to the famous example of deformations of the Mukai-Umemura threefold studied in detail by Tian \cite{0892.53027} and Donaldson \cite{1161.53066}. Since toric Fano manifolds are rigid \cite{0910.14004}, we cannot hope for such an example in the toric world.

Hence, we consider an effective action of an algebraic torus $T=(\CC^*)^m$ on a normal algebraic variety $X$, possibly of higher dimension than $T$. The union of the maximal orbits will be denoted by $\Xn$. We also consider some rational quotient map $\pi: X \dashrightarrow Y$. The dimension of $Y$ equals the \emph{complexity} of the torus action which is defined as the difference of dimensions $\dim X -\dim T$. Let $\mathcal{N}(T)$ be the normalizer of the acting torus inside $\Aut(X)$. Then one obtains an action of $\mathcal{N}(T)$ on $T$ by conjugation, which descends to an action on the characters  $M=\Hom(T,\CC^*)$. 
Assume that $G$ is a finite subgroup of $\mathcal{N}(T)$. This induces a $G$-action on $M$. Following \cite{0939.32016} the variety $X$ is called \emph{symmetric} with respect to the $T$-action if the identity is the only fixed point of such a $G$-action on $M$ for some finite subgroup $G \subset \mathcal{N}(T)$.

An equivariant automorphism $\varphi \in \mathcal{N}(T)$ descends to an automorphism $\underline{\varphi}$ of $\im(\pi)$, via $\underline{\varphi}(y):=\pi(\varphi(x))$ where $x$ is an arbitrary element in the preimage of $y$. 
Hence, we also obtain an action of $G$ on $\im(\pi)$, which is assumed to extend to the whole quotient $Y$. A fiber of $\pi$ is called \emph{non-reduced}, if $T$ acts with disconnected stabilizer on this fiber (indeed the fibers \emph{are} non-reduced in the algebro-geometric sense).

Our first result is the following criterion for the existence of K\"ahler-Einstein metrics on $T$-varieties of complexity one.
\begin{maintheorem}
\label{sec:thm-cplx-1}
  Let $X$ be a symmetric log terminal Fano $T$-variety of complexity one. If one of the following conditions is fulfilled:
  \begin{enumerate}
  \item there are three non-reduced fibers,
  \item there are two non-reduced fibers which are swapped by an element of $G$,
  \item $G$ acts fixed-point-free on $Y=\PP^1$,
  \end{enumerate}
then $X$ is K\"ahler-Einstein.
\end{maintheorem}

This result is used to prove the existence of a K\"ahler-Einstein metric on a certain blowup of the quadric threefold in two conics and on the hypersurface $V(xu^2+yv^2+zw^2) \subset \PP^2\times\PP^2$ of bidegree $(1,2)$. These give examples in the flavor of the Mukai-Umemura threefold \cite{0526.14006}, i.e. inside a family of Fano threefolds there is an element of high symmetry admitting a K\"ahler-Einstein metric and there are nearby elements with and without such a metric, see Section~\ref{sec:examples}.

The main tools for proving Theorem~\ref{sec:thm-cplx-1} are global log canonical thresholds (glct) of varieties and pairs, which are the algebro-geometric counterparts of Tian's $\alpha$-invariants. We are using the following theorem by Tian \cite[Theorem~2.1]{0599.53046}.  
\begin{ttheorem}
  Let $X$ be a smooth (or orbifold) Fano variety and $G \subset \Aut(X)$ a finite symmetry group. If we have the bound \[\lct_G(X) > \frac{\dim(X)}{\dim(X)+1},\]  for the global log canonical threshold, then $X$ is K\"ahler-Einstein.
\end{ttheorem}

For a symmetric T-variety the global log canonical thresholds on $X$ can be calculated as global log canonical thresholds of a pair on the quotient variety. \\For a prime divisor $Z$ on $Y$ we consider its preimage $\pi^{-1}(Z)$. Here, the generic stabilizer on a component of $\pi^{-1}(Z)$  will be a finite abelian group $\ZZ/m\ZZ$. The maximal order obtained on the components is denoted by $m_Z$. This gives rise to a boundary divisor $B=\sum_Z \frac{m_Z-1}{m_Z} \cdot Z$ on $Y$. 

\begin{maintheorem}
\label{sec:thm-main}
  If $X$ is a symmetric log terminal Fano T-variety 
and $\pi|_{\Xn}$ is a morphism onto $Y$, then we have
\[\lct_G(X) = \min\{1,\lct_G(Y, B)\}.\]
In particular, $\lct_G(X) \geq 1$ holds if and only if $\lct_G(Y, B)\geq 1$ holds.
\end{maintheorem}

  If $Y$ has only orbifold singularities, then the pair $(Y,B)$ can be interpreted as a kind of orbifold quotient of $X$. Moreover, the above result somehow motivates the following conjecture.
\begin{conjecture}
Assume that $X$ fulfills the preconditions of Theorem~\ref{sec:thm-main} and the quotient $(Y,B)$ admits an (orbifold) K\"ahler-Einstein metric, then $X$ is K\"ahler-Einstein.
\end{conjecture}

\begin{remark}
  The converse of the conjecture is not true as the quadric threefold shows. The quotient is $(\PP^1, \nicefrac{1}{2}\cdot\pt)$ which is not K\"ahler-Einstein by \cite{ross-thomas}. 
\end{remark}

 $G$ will be called \emph{\sufficient}for a pair $(X,B)$ if $\lct_G(X,B) \geq 1$. In particular, Tian's criterion states the existence of a K\"ahler-Einstein metric for Fano varieties admitting a \sufficient subgroup of symmetries. By Theorem~\ref{sec:thm-main} This property is preserved by our quotient maps.

\begin{remark}
  Also for another reason it makes sense to look for K\"ahler-Einstein metrics on symmetric Fano varieties. Since the Futaki character corresponds to an element of $M$ which is fixed by torus equivariant automorphisms, it has to vanish on a symmetric manifold. Hence, Futaki's necessary criterion is automatically fulfilled.
\end{remark}

\subsection*{Acknowledgment} I would like to thank Carl Tipler and J\"urgen Hausen for answering patiently all my questions on deformations of K\"ahler-Einstein metrics and GIT limits respectively. Moreover, I have to thank Nathan Ilten for sharing his knowledge on deformations of Fano varieties and for providing the excellent {Macaulay2} package \emph{VersalDeformations} \cite{versalDeform}.

\section{Log canonical thresholds}
\label{sec:lct}
First, we shortly remind the notion of log pairs and some of their properties. A log pair $(Y,B)$, consists of a normal variety $Y$ and a $\QQ$-divisor $B$, which is called \emph{boundary}. A canonical divisor of such a pair is simply the sum $K_X + B$ of a canonical divisor $K_Y$ on $Y$ and the boundary. Hence, we call a pair $\QQ$-Gorenstein if $K_Y + B$ is $\QQ$-Cartier. A log pair is called \emph{smooth}, if $Y$ is smooth and the boundary is a simply normal crossing divisor. A \emph{log resolution} of $(Y,B)$ is a birational proper morphism $\varphi: \widetilde{Y} \rightarrow Y$ such that $(\widetilde{Y},\varphi^*B)$ is smooth. A pair is called \emph{orbifold}, if $Y$ has only quotient singularities and $B$ is of the form $\sum_{i=1}^\ell \frac{m_i-1}{m_i}\cdot Z_i$, with positive integers $m_i$.

\begin{definition} Consider a Gorenstein pair $(Y,B)$ and a log resolution $\varphi:\widetilde{Y} \rightarrow Y$. Then there are canonical divisors $K_{\widetilde{Y}}$, $K_Y$ on $\widetilde{Y}$ and $Y$, respectively, such that $E:=K_{\widetilde{Y}} - \varphi^*(K_Y + B)$, is supported only on the strict transform of $\supp B$ and the exceptional divisor. The pair is called \emph{log canonical (l.c.)} if all the coefficients of $E$ are at least $-1$.

The coefficients of $E$ are called \emph{discrepancies} of the morphism $\varphi$ with respect to the pair $(Y,B)$.
\end{definition}

Let us denote the set of effective $\QQ$-divisors being linearly equivalent to a divisor $D$ by $|D|_{\QQ}$ and by $|D|_\QQ^G$ the set of those divisors which are also $G$-invariant. Having this we are going to introduce the central notion of this article.
\begin{definition}
\label{sec:def-lct}
  For a $\QQ$-Gorenstein pair $(Y,B)$ and a finite subgroup $G \subset \Aut(Y,B)$  one defines its \emph{global log canonical threshold} by
\ifdef{\ebook}{
  \begin{dmath*}[breakdepth={0},compact]
    \lct_G(Y,B) = \sup\left\{\lambda \;|\; (Y, B  + \lambda D)  \text{ is l.c. for all } D \in  |-K_X-B|_\QQ^G \right\}.
  \end{dmath*}
}{
\[ \lct_G(Y,B) = \sup\left\{\lambda \;|\; (Y, B  + \lambda D)  \text{ is l.c. for all } D \in  |-K_X-B|_\QQ^G \right\}.\]
}
\end{definition}

\begin{remark}
 Since \cite{0994.32021} appeared it is known, that global log canonical thresholds coincide with Tian's $\alpha$-invariant from \cite{0599.53046}. For a proof see Demailly's appendix of \cite{1167.14024}.
\end{remark}

Now, we will slightly generalize the notion of a pair and of a \sufficient subgroup of symmetries by allowing also $-\infty$ as a coefficient for boundary divisors. In this case we denote the open subset $Y \setminus \{-\infty \text{-locus of } B \} \subset Y$ by $Y^\circ$.
\begin{definition}
  A finite subgroup $G \subset \Aut(Y,B)$ is called \emph{\sufficient}for $(Y,B)$ if $(Y, D)$ is log canonical for all G-invariant $\QQ$-divisors $D$, with $D \sim_\QQ -K_Y$ and $D \geq B$.
\end{definition}

\begin{remark}
  For $B$ finite this definition simply means that $\lct_G(Y,B) \geq 1$, i.e. the definition coincides with the original one from the introduction.

  The notion of $(-\infty)$-coefficients looks a bit artificial. To illustrate the meaning we consider a $G$-invariant blowup $Y' \rightarrow Y$ with exceptional divisor $E$.  Now, $G$ is \sufficient for $(Y',-\infty \cdot E)$ if and only if the same is true for $Y$. 
\end{remark}

The notion of $(-\infty)$-coefficients allows us to generalize Theorem~\ref{sec:thm-main} to the case that $\pi|_{\Xn}$ is not a morphism onto $Y$.
\begin{theorem}
\label{sec:thm-sufficient}
  $G$ is {\sufficient} for a symmetric T-variety $X$ if and only if the same is true for the quotient $(Y,B)$.
\end{theorem}

\begin{example}
\label{sec:exmp-proj-line}
  Consider $X=\PP^1$ and $B=\nicefrac{1}{2}\cdot \{0\}+ \nicefrac{1}{2}\cdot \{\infty\}$. The effective $\QQ$-divisors $D \sim_\QQ -(K_{\PP^1}+B)$ are exactly those of degree $1$. Hence, the maximal possible coefficient of $(B + \lambda D)$ is $(\nicefrac{1}{2}+\lambda)$ and the pair $(\PP^1,B + \lambda D)$ is log canonical if and only if $\lambda \leq \nicefrac{1}{2}$. One obtains $\lct_{\langle 1 \rangle}(\PP^1,B)=\nicefrac{1}{2}$.  

Let $G$ be the group consisting of the identity and the involution given by $x \mapsto \nicefrac{1}{x}$. Then $G$ has exactly two fixed points: $1$ and $-1$. All other orbits consist of two elements. Let's assume that $D$ is $G$-invariant and chosen in such a way, that the biggest coefficient is maximal. Then $D$ might be either concentrated at one of the fixed points or it might be equally distributed over $x$ and $\nicefrac{1}{x}$, for some point $x \in \PP^1$. In the first case the maximal exponent of $(B + \lambda D)$ is $\max\{\nicefrac{1}{2},\lambda\}$ and in the second case it is $\nicefrac{\lambda}{2}+ \nicefrac{1}{2}$. In both cases $(\PP^1,B + \lambda D)$ is log canonical if and only if $\lambda \leq 1$. Hence, one obtains $\lct_G(\PP^1,B)=1$.
\end{example}

\begin{proposition}
  For a $T$-variety the global log canonical threshold can be calculated by considering only $T$-invariant divisors $D$.
\end{proposition}
\begin{proof}
  We argue exactly as in \cite[Section~5]{1167.14024}:  a priori the value in Definition~\ref{sec:def-lct} may grow when considering only torus invariant divisors. Let us assume that there is an effective divisor $D$ such that $(X,\lambda D)$ is not log canonical. Then we may choose a one-parameter subgroup $\gamma:\CC^* \hookrightarrow T$ and consider $D_0^\gamma=\lim_{t \rightarrow 0} \gamma(t).D$. Now, we have
\begin{enumerate}
\item $\gamma(t).D \sim D$ and $D^\gamma_0 \sim D$,
\item $(X, \lambda \cdot \gamma(t).D)$ is isomorphic to $(X,\lambda D)$, hence not log canonical, 
\item $D_0^\gamma$ is invariant by the $\CC^*$-action induced by $\gamma$.
\end{enumerate}
From \cite[Theorem~0.2.]{0994.32021} it follows, that $(X,\lambda D_0^\gamma)$ is not log canonical. After iteration with a basis of the group of one-parameter subgroups of $T$ we end up with a divisor $D_0$ which is $T$-invariant and $(X,\lambda D_0)$ is still not log canonical. Moreover, since the $G$ action normalizes the torus action $g.D_0$ will be torus invariant as well and linearly equivalent to $D$ for every element $g \in G$. Hence, the $\QQ$-divisor $\nicefrac{1}{|G|}\cdot (G.D_0)$ will be $(G \times T)$-invariant and again linearly equivalent to $D$.
\end{proof}

\section{Torus quotients and invariant divisors}
\label{sec:tvars}
For a semi-projective T-variety (i.e. projective over some affine variety) we obtain a quotient map by considering all the GIT-quotients corresponding to linearizations of an ample line bundle. These form an inverse system. The corresponding inverse limit comes with a rational map $X \dashrightarrow {\displaystyle \lim_{\longleftarrow}}\,Y_{u}$. The limit may be reducible and non-normal, but in any case we might consider the normalization of the distinguished component, which is defined as the closure of the image of this map. For a detailed discussion of this construction, see \cite{pre05013675} for the affine case and \cite{0762.14023,1147.14028} for the toric situation. We will denote the normalization of the  distinguished component simply by $X \QGIT T$. For projective varieties this construction coincides with the Chow quotient introduced by Kapranov, Sturmfels and Zelevinsky in \cite{0762.14023}, see \cite{0762.14023,1147.14028,1203.3759}. Hence, we will refer to the normalization of the distinguished component simply as the Chow quotient of $X$, even if the variety is not projective (e.g. affine).

\begin{remark}
  In general it is not easy to obtain an explicit description of $X \QGIT T$.
 Nevertheless, in important cases the quotient is well known, e.g. for the case of rational $T$-varieties of complexity one it has to be $\PP^1$. Even in higher complexity we sometimes obtain simply a projective space (e.g. for quadrics) or other well known varieties, as the moduli space $\overline{M}_{0,n}$  for the Grassmanian $G(2,n)$, see \cite{0811.14043}. Moreover, for toric varieties $X_\Sigma$ the Chow quotient with respect to a subtorus action $T' \subset T$ is given as the coarsest common refinement of the cones in the image $P(\Sigma)$, where $P:N \rightarrow N''$ is given by the  exact sequence
\[0 \longrightarrow N' \longrightarrow N \longrightarrow N'' \longrightarrow 0\]
of co-character lattices of the tori $T$, $T'$ and $T''=T'/T$, see \cite{0762.14023, pre05013675,1147.14028}.
\end{remark}

 For the Chow quotient map $\pi: X \dashrightarrow Y=X\QGIT T$ we have the following properties:
\begin{enumerate}
\item $\pi$ is $T$-invariant, $\pi^*K(Y)=K(X)^T$ holds,
\item $\pi|_\Xn$ is locally a geometric quotient onto its image. In particular, prime divisors on $\Xn$ are mapped to prime divisors on $Y$,
\item there is a natural $G$-action on $Y$ making $\pi$ a $G$-equivariant map,
\item for $U \subset X$ open and affine there is a semi-projective open subset $V \subset Y$ and a projective morphism $V \twoheadrightarrow  U\QGIT T$,
\end{enumerate}

We will also consider $G$-invariant, birational and projective modifications $Y \rightarrow X \QGIT T$ together with the induced maps $\pi:X \dashrightarrow Y$.  We call such a map simply a \emph{rational quotient map} of $X$. All the properties mentioned above hold for such a $Y$ as well. In the following we may assume that $Y$ is chosen to be at least $\QQ$-factorial.

\begin{remark}
  In \cite{pre05013675} Altmann and Hausen introduced a description of affine T-varieties by so called \emph{polyhedral divisors} on a quotient variety $Y$ as above. Property (iv) of the Chow quotient implies that in the projective case we may cover $X$ by affine open subsets corresponding to polyhedral divisors all living on $X\QGIT T$ (this leads to the notion of a \emph{divisorial fan} defined in \cite{divfans}). We will use this fact to apply several results from \cite{pre05013675,tidiv,tsing,tvars} which are formulated in the language of polyhedral divisors.
\end{remark}

\bigskip

We characterize two different types of $T$-invariant prime divisors on $X$:
\begin{itemize}
\item Those which intersect $\Xn$ are called \emph{vertical}, their images under $\pi$ are prime divisors in $Y$ (after taking the closure).
\item The prime divisors inside $X \setminus \Xn$ are called \emph{horizontal}, their image under $\pi$ is dense in $Y$.

\end{itemize}
The set of vertical divisors with image $Z \subset Y$ is denoted by $\Vert_Z$, the set of all vertical prime divisors by $\Vert$ and the set of all horizontal divisors by $\Hor$. Vertical prime divisors consist of closures of maximal orbits. Horizontal ones consist of orbit closures of dimension $(\dim T -1)$. On maximal orbits we have finite stabilizer groups. Moreover, the order of the stabilizers is fixed on an open subset of a vertical prime divisor $D$. This generic order is denoted by $\mu(D)$, and will be called simply the \emph{order of $D$}. The maximal order for all vertical prime divisors in $\Vert_Z$ is called \emph{multiplicity of $Z$} and will be denoted by $m_Z$. Note, that due to the effectiveness of the torus action there are only finitely many prime divisors $Z$ with multiplicity $>1$. For a prime divisor $Z$ on $Y$ we have then the following pullback formula, see \cite[Section~7]{tvars}:
\begin{equation}
\label{eq:1}
  \pi^*Z =  \sum_{D \in \Vert_Z} \mu(D)\cdot D.
\end{equation}

In particular, the pullback of a divisor on $Y$ does not have a horizontal component.
\begin{proposition}[{\cite[Section~8]{tvars}}]
\label{sec:K-formula}
  Let $K_Y$ be a canonical divisor on $Y$. Then 
\[\pi^*K_Y - \sum_{D \in \Hor} D + \sum_{D \in \Vert} (\mu(D)-1) \cdot D\]
defines a canonical divisor on $X$.
\end{proposition}

\begin{example}
   Consider the $\CC^*$-action on $\PP^2$ given by 
   \[t.(x_1:x_2:x_3)=(tx_1:tx_2:x_3).\]
   In this case the quotient is given by the map
   \[\pi:\PP^2 \dashrightarrow \PP^1;\quad (x_1:x_2:x_3) \mapsto (x_1:x_2).\] It is defined outside the point $(0:0:1)$. For every point $y=(a:b) \in P^1$ the set
$\Vert_{(a:b)}$ consists only of the invariant line $\{(a:b:s) \mid s \in \CC\}$ and outside of $(a:b:0)$ the stabilizer is trivial. Hence, $m_y=1$. There is exactly one horizontal prime divisor, namely the line $[x_3=0]$ consisting only of fixed points. 
\end{example}

\begin{example}[A hypersurface of bidegree $(1,2)$]
\label{sec:exmp-hypersurface}
  The hypersurface \[X=V(u_0v_0^2+u_1v_1^2+u_2v_2^2) \subset \PP^2 \times \PP^2\] of bidegree $(1,2)$ is a smooth Fano threefold. It admits a 2-torus action given by the weight matrix
\[
\begin{array}{rrrrrrrl}
  &u_0&u_1&u_2&v_0&v_1&v_2&
\vspace{2mm}\\
 \ldelim({2}{0.5ex}&2&0&0&-1&0&0&\rdelim){2}{0.5ex}\\
  &0& 2&0&0&-1&0&
\end{array}
\]
The quotient map is given by
 \[\pi:X\dashrightarrow \PP^1;\quad (u_0:u_1:u_2,\; v_0:v_1:v_2) \mapsto (u_0v_0^2:u_1v_1^2).\]
There are no horizontal prime divisors. For $y=(0:1)$ the vertical prime divisors $\Vert_y$ are the hyperplanes $[u_0=0]$ and $[v_0=0]$ having generic stabilizer of orders $1$ and $2$, respectively. Hence, one obtains multiplicity $m_y=2$. One obtains the same for the points $(1:0)$ and $(1:-1)$. For all other points $y=(a:b) \in \PP^1$ the unique vertical prime divisor in $\Vert_y$ is $[bu_0v_0^2-au_1v_1^2=0]$, having a trivial generic stabilizer. 

We have an action of $S_3\subset \mathcal{N}(T)$ on $X$ given by permuting the indices $0$, $1$, $2$ of the variables $u_i$ and $v_i$. Look at the cyclic permutation $0 \mapsto 1 \mapsto 2 \mapsto 0$. The induced lattice homomorphism is given by the matrix
$\left(\begin{smallmatrix}
  0 & -1\\
  1 & -1
\end{smallmatrix}\right)$
which has not real eigenvalues. Hence,  $X$ is symmetric and from Theorem~\ref{sec:thm-cplx-1} it follows the existence of a K\"ahler-Einstein metric, since the fibers over $0$, $\infty$ and $-1$ are non-reduced. 

The action of $S_3$ on $\PP^1$ is induced by the permutations of the points $0$, $-1$, $\infty$.
\end{example}

\section{Toroidal resolutions and discrepancies}
\label{sec:toroidal-res}
The definition of log canonicity of a T-variety $X$ a priori enforces to consider a log resolution, but it turns out that it is enough to resolve to a locally toric situation. For this we consider log-resolutions $\psi: \widetilde{Y} \rightarrow Y$ for a pair $(Y,\Delta)$ on the quotient $Y$ and lift it to a so called \emph{toroidal resolutions} $\widetilde{X}$ of $X$. Such a construction is given in \cite{tsing}.

 Given a $T$-pair $(X,D)$ with a rational quotient map $X \dashrightarrow Y$. By projecting the vertical part of $D$ this gives rise to a divisor on $Y$
 \[\supp_Y D := \pi(\supp_{\text{vert}} D).\]

\begin{proposition}
\label{sec:prop-toriodal-res}
  Given an affine $T$-pair $(X,D)$ and a rational quotient map $X \dashrightarrow Y$. Then, there are a birational $T$-invariant projective morphism $\varphi:\tX \rightarrow X$, an open subset $U \subset \tX$ and a log resolution $\tY \rightarrow Y$ of $(Y, \supp_Y D)$ having the following properties
  \begin{enumerate}
  \item There is a good quotient morphism $\widetilde{\pi}:\tX \rightarrow \tY$ fitting into the following commutative diagram
\[
\xymatrix{
\widetilde{X} \ar^{\widetilde{\pi}}[r] \ar^{\varphi}[d] & \tY  \ar^{\psi}[d]   \\
X  \ar^{\pi}@{-->}[r]  & Y 
}
\]
In particular, we have $\tX \QGIT T = \tY$.
  \item $(\tX,U)$ is an toroidal embedding, i.e. at every point $x \in \tX$ there is a toric variety $Z$ with torus $H$, such that $(\tX,U)$ is locally formally isomorphic to $(Z,H)$ at some point $z \in Z$,
  \item $\varphi|_U$ is an isomorphism onto its image,
  \item $\varphi^*D$ is supported on $\tX \setminus U$. Hence, locally it looks like a toric divisor in a toric variety $Z$.
  \end{enumerate}
\end{proposition}

The corresponding construction is given in Section~2 of \cite{tsing}, by using the language of polyhedral divisors. Using the results of \cite{2012arXiv1202.5760C} we can rephrase it as follows.

We resolve the indeterminacy of $\pi:X \dashrightarrow Y$ by considering the normalization of the closure of the graph of $\pi$ in $X \times Y$. This is again a T-variety. We denote it by $W$. The projection to the second factor gives a good quotient morphism $W \rightarrow Y$, which is a trivial fibration outside a subset $\Delta$ of pure codimension one. Now, we choose a log resolution $\tY \rightarrow Y$ of $(Y,\Delta \cup \supp_Y D)$ and denote the normalization of the closure of the graph of $X \dashrightarrow \tY$ in $X \times \tY$ by $\tX$. The projections to $X$ and $\tY$ induce the morphisms $\varphi$ and $\widetilde{\pi}$, respectively.

The following lemma shows that it is enough to resolve to such a locally toric situation for checking the log canonicity of a $T$-pair $(X,D)$, consisting of a $T$-variety and a $T$-invariant divisor.
\begin{lemma}
  The pair $(X,D)$ is log canonical if and only if the discrepancies of $\psi$ are at least $-1$, i.e. we have canonical divisors on $X$ and $\tX$, such that
\[K_{\tX} = \psi^*(K_X + D) + E \]
holds, with $E$ being supported only on $\supp D$ and the exceptional divisor of $\psi$ and the coefficients of $E$ are at least $-1$.
\end{lemma}
\begin{proof}
  One direction is obvious. Let $\theta$ be a log resolution of $(\tX,-E)$. Then the discrepancies of $\theta$ with respect to $(\tX,-E)$ are the same as the discrepancies of $\theta \circ \varphi$ with respect to $(X,D)$. Hence, for the other direction we have to show $(\tX, -E)$ is log canonical provided that the coefficients of $E$ are at least $-1$. But this is well known (and can easily be checked) for toric pairs. Since, we are locally formally in a toric situation the claim follows.
\end{proof}

In general we consider ``twiddled'' versions of our objects defined on $X$, i.e.
we will denote the horizontal and vertical prime divisors of $\tX$ by $\widetilde{\Hor}$ and $\widetilde{\Vert}$, respectively. The set of maximal orbits in $\tX$ will be denoted by $\tXn$.

\section{Comparing thresholds}
\label{sec:comparing-thresholds}
Remember, that we defined a boundary divisor $B=\sum_Z \frac{m_Z-1}{m_Z} \cdot Z$ on $Y$, where $m_Z$ is the maximal order of a generic stabilizer on a vertical divisor over $Z$. If there are no such divisors we set $m_Z=0$ and obtain an $-\infty$-coefficient for the boundary at $Z$. $Y^\circ$ was defined as the open set $Y \setminus \bigcup_{\Vert_Z=\emptyset} Z$.

\begin{lemma}
\label{sec:lemma-invariant-canonicals}
  Let $X$ be a symmetric T-variety and $G$ the corresponding finite subgroup $G \subset \mathcal{N}(T)$. Then  $(T\times G)$-invariant $\QQ$-divisors $D$, with $Q_X \sim_\QQ -K_X$ are exactly those of the form
  \begin{equation}
    \label{eq:2}
    Q_X=\pi^*Q_Y + \sum_{D \in \Hor} D + \sum_{D \in \Vert} (1-\mu(D)) \cdot D
  \end{equation}
  where $Q_Y$ is an  $G$-invariant divisor on $Y$ with $Q_Y\sim_\QQ -K_Y$.

  Moreover, $Q_X$ is effective iff $Q_Y \geq B$.
\end{lemma}
\begin{proof}
  Due to Proposition~\ref{sec:K-formula} the divisor $Q_X$ defined in (\ref{eq:2}) is linearly equivalent to $-K_X$. It is obviously $T$-invariant. The summand $\pi^*Q_Y$ is a $G$-invariant, since the $G$-action on $Y^\circ$ is defined via $\pi$. Horizontal prime divisors are necessarily mapped to horizontal ones by any equivariant automorphism. Hence, $\sum_{E \in \Hor} E$ is $G$-invariant. Similarly, vertical prime divisors are mapped to vertical ones. Moreover, the order of a vertical prime divisor is an invariant of the $T$-action on $X$. Hence, $\sum_{D \in \Vert} \left(1-\mu(D)\right) \cdot D$ is $G$-invariant, as well.

For the opposite direction consider a $(T \times G)$-invariant principal divisor $\divisor(f)$ on $X$. Then $f$ is a semi-invariant rational function, i.e. $f(t.x) = u(t) \cdot f(x)$ holds for some character $u \in M$, i.e. $\deg(f)=u$. Moreover, we may assume that $f$ is $G$-invariant, in particular the character $u$ has to be $G$-invariant. Now, the symmetry condition on $X$ implies that $u=0$ and $f \in K(X)^T=\pi^*K(Y)$. Hence, $\divisor(f) = \pi^*\divisor(g)$ holds for some principal divisor $\divisor(g)$ on $Y$. Therefore, if some divisor $D$ differs from $Q_X$ in (\ref{eq:2}) only by a divisor having a principal multiple, then $D$ is also of the desired form.

$Q_Y \geq B$ holds if and only if $Q_Y-B$ is effective on $Y^\circ$. This is equivalent to the effectiveness of $\pi^*(Q_Y-B)$. By the pullback formula (\ref{eq:1}) one obtains 
\begin{align*}
  \pi^*(Q_Y-B) &= \pi^*Q_Y  + \sum_Z \sum_{D \in \Vert_Z} (1-m_Z) \cdot D \\
              &\geq \pi^*Q_Y  + \sum_Z \sum_{D \in \Vert_Z} (1-\mu(D)) \cdot D\\
              &= \pi^*Q_Y  + \sum_{D \in \Vert} (1-\mu(D)) \cdot D
\end{align*}

Hence, $Q_Y \geq B$ implies the effectiveness of $Q_X$. For one of the prime divisors $D \in \Vert_Z$ we have $\mu(D)=m_Z$. Therefore, we also obtain the opposite direction. 
\end{proof}

\begin{proof}[Proof of Theorem~\ref{sec:thm-main}]
  At the beginning we introduce some notation.
  \begin{itemize}
  \item We consider the toroidal resolution and the maps from Proposition~\ref{sec:prop-toriodal-res}.
  \item  As before we consider  $G \times T$-invariant effective divisors $Q_X \sim_\QQ -K_X$.
  \item We set \[B_X=\sum_{D \in \Vert} (\mu(D)-1) \cdot D, \qquad B_{\widetilde{X}}=\sum_{\tD \in \widetilde{\Vert}} (\mu(\tD)-1) \cdot \tD\]
  \item We may choose some representations of the canonical classes on $\tX$ and $X$, respectively:
\[K_X=\pi^*K_Y + B_X- \sum_{D \in {\Hor}} D, \qquad K_{\tX}=\widetilde{\pi}^*K_\tY + B_\tX  - \sum_{\widetilde{D} \in \widetilde{\Hor}} \widetilde{D}\]
  \end{itemize}

As above we also choose a $G$-invariant effective divisor $Q_X\sim_\QQ -K_X$. Then we have $Q_X = \pi^* Q_Y - B_X$, with $(Q_y - B) \in |-K_Y-B|_\QQ^G$ by Lemma~\ref{sec:lemma-invariant-canonicals} and $\pi^*B \geq B_X$. If $\lambda \leq 1$ one obtains for $E_\tX:= K_{\tX} - \varphi^*(K_X + \lambda Q_X)$ on $\tXn$
  \begin{align}
    E_\tX &= K_{\tX} - \varphi^*(K_X + \lambda Q_X)\nonumber \\
    &=  \widetilde{\pi}^*K_{\tY} +B_\tX - \varphi^*(\pi^*K_Y +B_X + \lambda(\pi^*Q_Y- B_X))\nonumber \\
    &= \widetilde{\pi}^*K_{\tY} +B_\tX - \varphi^*(\pi^*K_Y +(1-\lambda)B_X + \lambda \pi^*Q_Y)\nonumber \label{eq:3}\\
    &\geq \widetilde{\pi}^*K_{\tY} +B_\tX - \varphi^*(\pi^*K_Y +(1-\lambda)\pi^*B + \lambda \pi^*Q_Y)\\
    &= \widetilde{\pi}^*(K_\tY - \psi^*(K_Y + B + \lambda (Q_Y - B))) + B_\tX \nonumber
  \end{align}

  Note, that here we used that $\pi$ is defined on $\Xn$ and surjective. Hence, $\tilde{\pi}^* \circ \psi^*$ and $\varphi^* \circ \pi^*$ coincide as homomorphisms $\wdiv_\QQ(Y) \rightarrow \wdiv_\QQ(\tXn)$. Now, assume, that $(Y,B + \lambda (Q_Y-B))$ is log canonical, then by definition  all the coefficients of $E_\tY :=K_\tY - \psi^*(K_Y + B + \lambda (Q_Y - B))$ are at least $-1$, but by our pullback formula (\ref{eq:1}) we infer that the coefficients of $\widetilde{\pi}^*E_Y + B_\tX$ are $\geq -1$ as well. If we replace $X^\circ$ by an open subset obtained by removing all vertical prime divisors $D$ which are not of maximal order in $\Vert_{\pi(D)}$, then we even have equality in (\ref{eq:3}). Hence, the coefficients of $E_\tX$ are $\geq -1$ if and only if the same is true for $E_\tY$.

It remains to check the discrepancies at the horizontal divisors. Remember, that we have $Q_X \sim_\QQ -K_X$, i.e. $(Q_X + K_X) =  \nicefrac{1}{m} \cdot \divisor(f)$ for some torus invariant function $f \in K(X)^T$. We get
\begin{align*}
 K_\tX - \varphi^*(K_X+ \lambda Q_X) &=  K_\tX - \varphi^*(K_X+Q_X + (\lambda-1)Q_X)\\
&= K_\tX - \varphi^*(\nicefrac{1}{m}\cdot \divisor(f)) + \varphi^*((1-\lambda)Q_X)\\
&\geq  K_\tX -\nicefrac{1}{m}\cdot\divisor(\varphi^*f)
\end{align*}

Note, that we assumed $\lambda \leq 1$, again. Now, $\varphi^*f$ is again a torus invariant function. Hence, $\divisor(\varphi^*f)$ is supported only on vertical divisors and for $K_\tX = \widetilde{\pi}^* K_\tY + B_\tX - \sum_{\tD \in \widetilde{\Hor}} \tD$  all the coefficients at horizontal divisors are $-1$.

It remains to show that we can indeed assume that $\lambda \leq 1$, i.e. $\lct(X,G) \leq 1$. First note, that the resolution $\tX \rightarrow \tY$ is a fibration into complete toric varieties. Let us denote the general fiber by $X'$ and an arbitrary invariant prime divisor of $X'$ by $D$. Now, there is a big open subset of the form $X' \times U \subset \tX$ with $U \subset \tY$. Now, $D \subset X'$ gives rise to a horizontal prime divisor $H=\overline{D \times U} \subset \tX$. In particular there exists a horizontal divisor $H$ on $\tX$. We choose $Q_X=K_X$. Now we have $K_\tX - \varphi^*(K_X+ \lambda Q_X) = K_\tX + (1-\lambda)\cdot\varphi^*K_X$. Consider a horizontal divisor $H$ on $\tX$. On the one hand then the coefficient of $K_\tX$ at $H$ is $-1$. On the other hand by the fact that $X$ is log terminal, we know that the coefficient of $K_\tX - \varphi^*K_X$ is bigger than $-1$. Hence, the coefficient of $\varphi^*K_X$ at $H$ is positive and  the coefficient of $K_\tX + (1-\lambda)\cdot\varphi^*K_X$  is $\geq -1$ if and only if $\lambda \leq 1$.
\end{proof}

\begin{proof}[Proof of the Theorem~\ref{sec:thm-sufficient}]
In general we do not have $(\pi \circ \varphi)^*D = (\psi \circ \widetilde{\pi})^*D$ on $\tX$, since functoriality does not hold for rational maps. Nevertheless, we have functoriality for rational functions, hence $(\pi \circ \varphi)^*D = (\psi \circ \widetilde{\pi})^*D$ holds if $D \sim_\QQ 0$. Therefore, on $\tXn$ we have
\begin{align*}
 K_\tX - \varphi^*(K_X+ Q_X) &=  K_\tX - \varphi^*(\pi^*(K_Y+Q_Y))\\
&= K_\tX - \widetilde{\pi}^*(\psi^*(K_Y+Q_Y))\\
&= \widetilde{\pi}^*(K_{\tY} - \psi^*(K_Y+Q_Y)) + B_\tX.
\end{align*}
By the pullback formula~(\ref{eq:1}) the coefficients of this divisor are $\geq -1$ if and only if this is true for the divisor $K_{\tY} - \psi^*(K_Y+Q_Y)$, but the latter is equivalent to the log canonicity of the pair $(Y,Q_Y)$. Hence, $G$ is \sufficient for $X$ if and only if the same is true for $(Y,B)$.
\end{proof}

\begin{proof}[Proof of the Theorem~\ref{sec:thm-cplx-1}]
  For the case of complexity one the quotient has to be $\PP^1$ (it has to be a curve and the Fano property implies a finitely generated Picard group) and the quotient map is indeed defined on the whole set $\Xn$ and automatically surjective. One easily checks (see e.g. Example~\ref{sec:exmp-proj-line}) that for $B=\sum_P \frac{m_P-1}{m_P}$ we have
  $\lct_G(\PP^1,B) \geq 1$ exactly in the following case
  \begin{enumerate}
  \item $\#\supp B \geq 3$
  \item $\#\supp B = 2$ and $G$ acts without fixed point on  $\supp B$
  \item $G$ acts without fixed point on  $\PP^1$
  \end{enumerate}
These are exactly the cases from Theorem~\ref{sec:thm-cplx-1}. Now, the claim follows from Tian's theorem.
\end{proof}

\section{Examples and K\"ahler-Einstein metric on their deformations}
\label{sec:examples}

\subsection{Hypersurfaces of bidegree $(1,2)$}
We come back to the manifold of Example~\ref{sec:exmp-hypersurface}, i.e. the hypersurface $V(xu^2+yv^2+zw^2) \subset \PP^2 \times \PP^2$ of bidegree $(1,2)$. By Theorem~\ref{sec:thm-cplx-1} this manifold admits a K\"ahler-Einstein metric. A semi-universal deformation of $X$ is given by the family 
\[\X_{\alpha,\beta,\gamma}=V(xu^2+yv^2+zw^2+\alpha\cdot xvw + \beta \cdot yuw + \gamma \cdot zuv)\]
over $(\CC^3,0)$. Note, that the action of $(\CC^*)^2$ on $X=\X_0$ naturally extends to an action on this family. Remember the weights for the action on $X$:
\[
\begin{array}{rrrrrrrl}
  &x&y&z&u&v&w&
\vspace{2mm}\\
 \ldelim({2}{0.5ex}&2&0&0&-1&0&0&\rdelim){2}{0.5ex}\\
  &0& 2&0&0&-1&0&
\end{array}
\]
Hence, the equation of the family has to be of weight $0$ and the weights for the action on the base are given by
\vspace{-1em}
\[
\begin{array}{rrrrl}
  &\alpha&\beta&\gamma&
\vspace{2mm}\\
 \ldelim({2}{0.5ex}&-2&1&1&\rdelim){2}{0.5ex}\\
  &1&-2&1&
\end{array}
\]
A point $(\alpha,\beta,\gamma)$ is \emph{polystable}, i.e. has a closed $(\CC^*)^2$-orbit under this action, if and only if $\alpha\beta\gamma \neq 0$. This can be seen by checking the existence of the limits \[\lim_{t \rightarrow 0} (t^{v_1},t^{v_2}).(\alpha,\beta,\gamma)\] 
for every one-parameter subgroup $t \mapsto (t^{v_1},t^{v_2})\in (\CC^*)^2$. Alternatively one can use the explicit combinatorial criterion in terms of balanced families of weights, given in  \cite[Proposition 3.3.3.]{1201.4137}.

Now, we have the following statement
\begin{proposition}
\label{sec:prop-deformations-hypersurface}
   For a sufficiently small neighborhood $0 \in \B \subset \CC^3$ the manifold $\X_{\alpha,\beta,\gamma}$ admits a K\"ahler-Einstein metric if and only if $\alpha\beta\gamma \neq 0$.
 \end{proposition}
 
To prove this behavior we are using the following theorem implicitly given in \cite{1201.4137}:
\begin{theorem}
\label{sec:deform-csck}
  Assume that a simply connected smooth $T$-variety $X$ admits a K\"ahler metric of constant scalar curvature (cscK metric) in the class $c(L)$ for an ample line bundle $L$. Consider the semi-universal deformation $\X \rightarrow \B$, where $0 \in \B \subset H^1(X,\T_X)$ is a sufficiently small neighborhood of $0$. If $H^2(X,\CO)=0$ then $L$ lifts to a line bundle $\L$ on $\X$\footnote{see e.g. \cite[3.3.11]{1102.14001}}.

If in this situation $\Aut^\circ(X)=T$ 
and  $H^2(X,\T_X)=0$ holds, then the manifold $\X_t$ admits a cscK metric in the class $c(\L_t)$ if and only if $t\in \B$ is a polystable point of the $T$-action on $H^1(X, \T_X)$.
\end{theorem}
  This corresponds to Theorem~1.2.3. in \cite{1201.4137}. There it is stated for the toric case, but the proof uses only the facts that toric varieties are simply connected and that $H^2(X,\CO_X)=0$ holds for every toric variety. Note, that the original condition of Theorem~1.2.3 is that the complexification of the group of Hamiltonian isometries $H^\CC$ is an algebraic torus. As it is pointed out later on in their paper, in our situation $H$ is a maximal compact subgroup of $\Aut^\circ(X)$. Now, the condition on the Hamiltonian isometries follows from ours: $\Aut^\circ(X)=T$.

\begin{proof}[Proof of Proposition~\ref{sec:prop-deformations-hypersurface}]
  By Kodaira vanishing we have conditions $H^2(X,\CO_X)=0$ and $H^2(X,\T_X)=0$ fulfilled. Moreover, since $X$ is Fano it is also simply connected. Since a K\"ahler-Einstein metric is a cscK metric in the anti-canonical class we may apply Theorem~\ref{sec:deform-csck} and the claim follows.
\end{proof}

\begin{remark}
  This is a similar situation as for the Mukai-Umemura threefold, where one finds a symmetric element (here admitting an $\Sl_2$-action) in the family $V_{22}$ of Fano threefolds which is K\"ahler-Einstein. In the neighborhood of this threefold the $\Sl_2$-symmetry is lost and one finds nearby-elements not being K\"ahler-Einstein. See \cite{0892.53027} and \cite{1161.53066} for a detailed discussion of this situation.
\end{remark}

\subsection{Torus invariant blowup of the quadric threefold}
Consider th quadric $Q^3=V(u_0^2+u_1v_1+u_2v_2) \subset \PP^4$. We have an action of the 2-torus given by the weights
\[
\begin{array}{rrrrrrl}
  &u_0&u_1&u_2&v_1&v_2&
\vspace{2mm}\\
 \ldelim({2}{0.5ex}&0&1&0&-1&0&\rdelim){2}{0.5ex}\\
  &0& 0&1&0&-1&
\end{array}
\]
We get the rational quotient map
\[\pi:Q \dashrightarrow \PP^1, \quad (u_0:u_1:u_2:v_1:v_2) \mapsto (u_1v_1:u_2v_2).\]
There is an equivariant involution $\sigma:Q \rightarrow Q$ induced by the change of variables $u_i \leftrightharpoons v_i$ for $i=1,2$. By looking at the weights of $u_i$ and $v_i$ we see, that this corresponds to the involution $w \leftrightharpoons -w$ on $M$. Hence, $Q$ is symmetric. 

Let us study the invariant prime divisors on $Q$. There are no horizontal prime divisors. Over $0$ and $\infty$, respectively, we find pairs of divisors cut out by $[u_i=0]$ and $[v_i=0]$ and having order $1$. Moreover, we there is also a vertical prime divisor of order $2$ over the point $-1$. Over every other point $(a:b)  \in \PP^1$ there is exactly one vertical prime divisor, having order $1$. Quadric hypersurfaces are K\"ahler-Einstein. Nevertheless, we cannot prove it by using Theorem~\ref{sec:thm-cplx-1}. We are not interested in $Q$ itself, but in the blowup of $Q$ in the two conics $C_1$ an $C_2$ cut out by $[u_i=v_i=0]$ for $i=1,2$. These conics are torus invariant and invariant by the involution. Hence, the resulting variety is again a symmetric $T$-variety. Moreover the quotient map is simply the composition of $\pi$ with the blowup morphism (there is no choice, the quotient has to be $\PP^1$ and the rational map has to coincide with $\pi$ outside the exceptional divisors). It follows that the exceptional prime divisors are vertical prime divisors over $0$ and $\infty$, respectively. Both have order $2$, since locally we are blowing up $\A^3 = \spec \CC[x,y,z]$ in the ideal $(y,z)$. Here, we may assume that $x=\nicefrac{u_0}{v_2}$, $y=\nicefrac{u_1}{v_2}$ and $z=\nicefrac{v_1}{v_2}$ with weights 
\[\small
\begin{array}{rrrrl}
  &x&y&z&
\vspace{2mm}\\
 \ldelim({2}{0.5ex}&0&1&0&\rdelim){2}{0.5ex}\\
  &1& 1& 2&
\end{array}.
\]
Locally the blowup is given as $\{(x,y,z\;;\,u:v)  \in \A^3 \times \PP^1 \mid yu-zv=0 \}$ and the extended torus action by the additional weights $\deg(u)={-1 \choose -1}$ and $\deg(v)={0 \choose -2}$. Now, the stabilizer at a generic point $(x,0,0\;;\,u:v)$ is $\{(1,1),(-1,-1)\} \subset T$.

By Theorem~\ref{sec:thm-cplx-1} we obtain the following
\begin{proposition}
  $\Bl_{C_1,C_2} Q$ is K\"ahler-Einstein.
\end{proposition}

Again we are interested in the deformation theory of the manifold. The semi-universal deformation of $X=\Bl_{C_1,C_2} Q^3$ is given by blowing up the family 
\[\X=V(u_0^2 + u_1v_1 + u_2v_2 + \alpha \cdot u_1v_2 + \beta \cdot u_2v_1 + \gamma \cdot u_1u_2 + \delta \cdot v_1v_2
)\]
in the subvarieties cut out by $[v_i=u_i=0]$, $i=1,2$. The 2-torus acts on the base by the following weights
\[
\begin{array}{rrrrrl}
  &\alpha&\beta&\gamma&\delta&
\vspace{2mm}\\
 \ldelim({2}{0.5ex}&-1&1&-1&1& \rdelim){2}{0.5ex}\\
  &1&-1&-1&1&
\end{array}
\]
The polystable points are exactly those points $(\alpha,\beta,\gamma,\delta)$, such that $\alpha \beta \neq 0$ or $\gamma \delta \neq 0$.  As in the previous section we can apply Theorem~\ref{sec:deform-csck} to obtain deformations with K\"ahler-Einstein metric as well as those without such metrics.

\begin{remark}
  Both constructions, that of blowups of quadrics and that of hypersurfaces of bidegree $(1,2)$ generalize well to higher dimensions. Here, the quotients to consider are projective spaces of higher dimensions. 
\end{remark}

\bibliographystyle{halpha}
\bibliography{../bib/all}
\end{document}